\documentclass{article}

\usepackage{amsmath}
\usepackage{url}
\usepackage{hyperref}

\RequirePackage{graphicx}
\RequirePackage{pifont,latexsym,ifthen,amsthm,rotating,calc,textcase,booktabs}
\RequirePackage{amsfonts,amssymb,amsbsy,amsmath}

\theoremstyle{definition}
\newtheorem{theorem}{Theorem}

\newtheorem{definition}[theorem]{Definition}

\newtheorem{conjecture}[theorem]{Conjecture}

\newcommand\acks{\section*{Acknowledgements}}

\makeatletter
\@twosidetrue
\makeatother

\begin{document}

\author{Jonathan Fine\\
15 Hanmer Road, Milton Keynes, MK6 3AY, United Kingdom}
\title{A note on braids and Parseval's theorem}
\date{22 November 2009}

\maketitle

\newtheorem{problem}[theorem]{Problem}

\begin{abstract}% Need this to prevent space sneaking in at start of abstract.
\noindent
In 1988 Falk and Randell, based on Arnol'd's 1969 paper on braids,
proved that the pure braid groups are residually nilpotent.  They also
proved that the quotients in the lower central series are free abelian
groups.

This brief note uses an example to provide evidence for a much
stronger statement: that each braid $b$ can be written as an infinite
sum $b =\sum_0^\infty b_i$, where each $b_i$ is a linear function of
the $i$-th Vassiliev-Kontsevich $Z_i(b)$ invariant of $b$.  The
example is pure braids on two strands.  This leads to solving
$e^\tau=q$ for $\tau$ a Laurent series in $q$. We set $\tau =
\sum_1^\infty (-1)^{n+1} (q^n - q^{-n})/n$ and use Fourier series and
Parseval's theorem to prove $e^\tau=q$.

For more than two strands the stronger statement seems to rely on an
as yet unstated Plancherel theorem for braid groups, which is likely
both to be both and to have deep consequences.
\end{abstract}

\noindent
Throughout we will let $P_n$ denote the group of pure braids on
$n$-strands and $\mathcal{P}_n$ the $L^2$ space of square-summable
functions $f:\mathcal{P}_n\to\mathbb{C}$.  We will think of $v
\in\mathcal{P}_n$ as a convergent infinite formal sum of elements of $P_n$.

For the convenience of the reader, we first restate the main result of
Falk and Randell \cite[pp221-2]{falk1986pure}, which relies on earlier
work of Arnol'd~\cite{arnold1969cohomology}.  In their words (p217),
they show that the lower central series of the pure braid group
$P_n$ is identical to that of an appropriate product of free groups.

\begin{theorem}[Residual nilpotence, Falk and Randall]
For any group $G$ the \emph{lower central series} $\{G_i\}$ is given
by $G_0=G$ and $G_{i+1}=[G_i, G]$, where $[A,B]$ denotes the subgroup
generated by commutators of elements of $A$ and $B$.  Let $G = P_n$.
Then $\bigcap_{i=1}^\infty G_i = \{1\}$.  In addition the quotients
$G_i/G_{i+1}$ are free abelian groups.
\end{theorem}

For $P_n$ to be the fundamental group of its configuration space $M_n$
a base point $(z_1, \ldots, z_n)$ must be chosen.  Once that is done,
each element $b\in P_n$ has Vassiliev-Kontsevich invariants $Z_i(b)$
lying in finite-dimensional value spaces $A_{n,i}$.  We can now state

\begin{conjecture}
\label{conj:b=sumbi}
Let $Z_i(b)$ be as above.  Then there are linear functions
$\psi_i:A_{n,i}\to \mathcal{P}_n$ such that $b = \sum_{i=0}^\infty
b_i$, where $b_i = \psi_i(Z_i(b))$.  In addition, $\psi_{i+j}(v_iv_j)
= \psi_i(v_i)\psi_j(v_j)$.
\end{conjecture}

The conjecture states that the \emph{inverse problem}, to that of
computing $Z_i(b)$ from $b$, has a solution.  We will now prove the
conjecture for $n=2$.  We have $P_2\cong\mathbb{Z}$.  Let $q$ denote a
generator of $P_2$ and $p = q^{-1}$ its inverse.  Because the volume
of the unit $n$-simplex is $1/n!$ it follows that $Z_n(q) = t^n/n!$
where $t=Z_1(q)$.  Here is the conjecture, applied to $q$.

\begin{problem}
Find a $\tau\in\mathcal{P}_2$ such that $q = 1 + \tau + \tau^2/2! +
\tau^3/3! + \ldots$.
\end{problem}

This is a shorthand for saying first that the convolutions $\tau,
\tau^2, \tau^3, \ldots$ all lie in $\mathcal{P}_2$ and second that the
sum $1 + \tau + \tau^2/2! + \ldots$ converges to $q\in \mathcal{P}_2$.
To prove this result we use Fourier series and Parseval's theorem.  We
write $\mathcal{P}_2$ as $L^2(\mathbb{Z})$.  The problem now is to
solve $e^\tau = q$ for $\tau$ a Laurent series in $q$.

We use a trick to obtain a candidate for $\tau\in L^2(\mathbb{Z})$.
Write $p=q^{-1}$.  We can write $q = (1+q)/(1+p)$ and so at least
formally our candidate is $\ln(1+q) - \ln(1+p)$.

\begin{definition}
\label{def:tau}
\begin{equation*}
\tau = \sum_1^\infty (-1)^{n+1}(q^n-p^n)/n \in \mathcal{P}_2
\end{equation*}
Because $\sum_1^\infty 1/n^2$ is absolutely convergent, $\tau$ is in
$\mathcal{P}_2$.  Note that $f(z) = \sum_1^\infty
(-1)^{n+1}(z^n-z^{-n})/n$ is nowhere absolutely convergent.
\end{definition}

The next result appears to be new.  Its proof is an exercise in
Fourier series which, for lack of a suitable reference, we give here.
The proof has an algebraic part and an analytic part.

\begin{theorem}
  \label{thm:q=exp(tau)}
  $\tau \in L^2(\mathbb{Z})$, as defined in Definition~\ref{def:tau},
  satisfies $\exp(\tau) = q$.
\end{theorem}

For any integrable function $f$ defined on $[-\pi, \pi]$ we as usual
let $ c_n(f) = 1/2\pi \int_{-\pi}^{\pi}
e^{-in\theta}f(\theta)\,d\theta $ denote the $n$-th complex Fourier
coefficient of $f$.  For the function $f(\theta)=\theta$ we have
\begin{equation}
\label{eqn:cn(f)}
  c_n(f) = \frac{1}{2\pi} \int_{-\pi}^\pi e^{-in\theta}\theta\, d\theta
  = \left.\frac{i}{2n\pi}e^{-in\theta}\theta\right|_{-\pi}^{\pi}
  - \frac{i}{2n\pi}\int_{-\pi}^\pi e^{-in\theta}\, d\theta
% TODO: make perfectly clear?
  =\frac{i(-1)^{n}}{n}
\end{equation}
for $n\ne 0$, while $c_0(f) = \int_{-\pi}^\pi\theta\,d\theta = 0$.
Thus, as a series $\tau$ is the Fourier transform of $i\theta$.

We can extend (\ref{eqn:cn(f)}) as follows.  For $\psi$ in
$L^2(\mathbb{Z})$ we use $c_n(\psi)$ to denote $\psi_n$, which we also
interpret as the coefficient of $q^n$.

\begin{theorem}
\label{thm:c_n(tau^m)}
\begin{equation*}
  c_n(\tau^m) = \frac{1}{2\pi}\int_{-\pi}^{\pi}\> e^{-in\theta}\>(i\theta)^m \> d\theta
\end{equation*}
\end{theorem}

%% For some reason, double bracing required here.
\begin{proof}[{{Proof of Theorem \ref{thm:q=exp(tau)}}}]
The algebraic part of the proof, which relies on
Theorem~\ref{thm:c_n(tau^m)}, is
\begin{align*}
  c_n(\exp(\tau)) &= \sum \frac{c_n(\tau_m)}{m!}\\
  &= \frac{1}{2\pi}\sum\int_{-\pi}^{\pi}e^{-in\theta}\frac{(i\theta)^m}{m!}\, d\theta\\
  &= \frac{1}{2\pi}\int_{-\pi}^{\pi}e^{-in\theta}\sum \frac{(i\theta)^m}{m!}\, d\theta\\
  &= \frac{1}{2\pi}\int_{-\pi}^{\pi}e^{-in\theta} e^{i\theta}\, d\theta\\
  &= \frac{1}{2\pi}\int_{-\pi}^{\pi}e^{i(1-n)\theta}\, d\theta
\end{align*}
and hence $c_1=1$ and $c_n=0$ otherwise.  The analytic part is that
the sum-integral is absolutely convergent and so, by Fubini's theorem,
we can perform the integration first (which then allows us to simplify
the sum).
\end{proof}

\begin{theorem}[Parseval's theorem]
Let $A(x)$ and $B(x)$ be integrable functions on $[-\pi, \pi]$ with
complex Fourier coefficients $a_n$ and $b_n$.  Then
$
\sum_{-\infty}^\infty
a_n\overline{b_n} = 1/2\pi\int_{-\pi}^\pi
A(x)\overline{B(x)}\,dx
$.
\end{theorem}

\begin{proof}[{{Proof of Theorem \ref{thm:c_n(tau^m)}}}]
We rewrite the result to be proved as
\begin{equation*}
  c_n(\tau^m) = \frac{1}{2\pi}\int_{-\pi}^{\pi}\> 
  (i\theta)^{m-1}
  \times
  (i\theta)
  \>
  e^{-in\theta}
  \>
  d\theta
\end{equation*}
and apply Parseval's theorem with $A=(i\theta)^{m-1}$ and $B=
\overline{i\theta e^{-in\theta}}$ (and an induction hypothesis). This
tells us that the right hand side is equal to
$
  \sum c_k(\tau^{m-1})\overline{c_k(B)}
$
and as
\begin{equation*}
  c_k(\overline{i\theta e^{-in\theta}}) 
  = \frac{1}{2\pi}\int_{-\pi}^{\pi}\overline{i\theta} e^{in\theta}e^{-ik\theta}d\theta
  = c_{n-k}(\tau)
\end{equation*}
the result follows.
\end{proof}

The concludes the proof that $e^\tau=q$.  The remaining steps required
to prove the rest of the conjecture are straightforward, and are left
to the reader.  A longer version of this note, together with
speculation on how the extending the conjecture to knots and proving
it is at~\cite{fine2009vassiliev}.  Proving the extension is probably
hard.

\acks 

I thank Phil Rippon for help with the proof of
Theorem~\ref{thm:q=exp(tau)}, and Joel Fine for reading an earlier
version of this paper. Any remaining errors are mine.

\bibliographystyle{amsplain}
\bibliography{braids}

\end{document}